\documentclass[12pt]{article}

\usepackage[a4paper,margin=1.12in]{geometry}
\usepackage[T1]{fontenc}
\usepackage[utf8]{inputenc}
\usepackage{lmodern}
\usepackage{microtype}
\usepackage{amsmath,amsthm,amssymb,mathtools}
\usepackage{enumitem}
\usepackage{booktabs}
\usepackage{array}
\usepackage{graphicx}
\usepackage{tikz}
\usetikzlibrary{arrows.meta,calc,decorations.pathreplacing,positioning}
\usepackage{xcolor}
\usepackage{hyperref}
\usepackage{setspace}
\allowdisplaybreaks

\hypersetup{
  colorlinks=true,
  linkcolor=blue!50!black,
  citecolor=blue!50!black,
  urlcolor=blue!50!black
}

\newtheorem{theorem}{Theorem}[section]
\newtheorem{proposition}[theorem]{Proposition}
\newtheorem{corollary}[theorem]{Corollary}
\newtheorem{lemma}[theorem]{Lemma}
\newtheorem{conjecture}[theorem]{Conjecture}
\newtheorem{problem}[theorem]{Problem}
\theoremstyle{definition}
\newtheorem{definition}[theorem]{Definition}
\theoremstyle{remark}
\newtheorem{remark}[theorem]{Remark}

\newtheorem{example}[theorem]{Example}
\newtheorem*{targettheorem}{Main Theorem}

\newcommand{\Ind}{\operatorname{Ind}}
\newcommand{\N}{\mathrm{N}}
\newcommand{\dist}{\operatorname{dist}}
\newcommand{\mode}{\operatorname{mode}}

\title{\textbf{Closing Trees into Unicyclic Counterexamples}\\[8pt]
\large Independence polynomials that stay unimodal but lose log-concavity}
\author{%
Vadim E.~Levit\textsuperscript{*}\qquad Ohr Kadrawi\textsuperscript{**}\\[6pt]
\small Department of Mathematics, Ariel University, Ariel 40700, \underline{ISRAEL}}
\date{March 17, 2026}

\begin{document}
\maketitle
\begingroup
\renewcommand{\thefootnote}{\ifcase\value{footnote}\or *\or **\else \arabic{footnote}\fi}
\footnotetext[1]{\href{mailto:levitv@ariel.ac.il}{\texttt{levitv@ariel.ac.il}}}
\footnotetext[2]{\href{mailto:orka@ariel.ac.il}{\texttt{orka@ariel.ac.il}}}
\endgroup
\setcounter{footnote}{0}
\renewcommand{\thefootnote}{\arabic{footnote}}

\begin{abstract}
We develop a family-based route to unicyclic graphs whose independence polynomials are unimodal but not log-concave. The paper is organized around one flagship statement: for the explicit KL-closure family $U_{k,r}$, with $r\in\{0,1,2\}$ and admissible $k$, the independence polynomial is unimodal but not log-concave. The proof separates the closure polynomial into a dominant convolution term and a real-rooted correction term. On the non-log-concavity side, we prove symbolically that the penultimate log-concavity inequality fails for every admissible parameter. On the unimodality side, we prove that the main convolution term $H_{k,r}=G_kF_{k+r}$ is unimodal with a controlled mode, using a combination of exact coefficient formulas, Ibragimov's strong-unimodality principle, and a residue-class growth argument. Darroch localization and an adjacent-mode bridge lemma then transfer that mode statement to the full KL closure polynomial. This yields an explicit infinite family of unicyclic graphs with unimodal but non-log-concave independence polynomials. In the exact range $k\le 400$ we further verify that the penultimate break is unique and determine exact mode formulas for $H_{k,r}$, the binomial correction term, and $I(U_{k,r};x)$ itself. The paper also places the KL family inside a broader reservoir program involving Galvin, Ramos--Sun, and Bautista--Ramos trees, from which we obtain substantial universal exact theorems for finite ranges.
\end{abstract}
\medskip
\noindent\textbf{Keywords.} independence polynomial; unimodality; log-concavity; unicyclic graph; one-edge closure; graph polynomial; independent sets.

\smallskip
\noindent\textbf{MSC 2020.} Primary 05C31, 05C69. Secondary 05C05, 05C38, 05C76, 05A20.

\section{Introduction}

For a finite simple graph $G$, let
\[
I(G;x)=\sum_{k=0}^{\alpha(G)} s_k(G)x^k
\]
denote its \emph{independence polynomial}, where $s_k(G)$ counts independent sets of size $k$ and $\alpha(G)$ is the independence number of $G$. A nonnegative sequence $(a_0,a_1,\dots,a_d)$ is \emph{unimodal} if it rises to a peak and then falls, and it is \emph{log-concave} if
\[
a_i^2\ge a_{i-1}a_{i+1}\qquad (1\le i\le d-1).
\]
Every positive log-concave sequence is unimodal. For general background we refer to Brenti~\cite{Brenti1989}; for the classical independence-sequence problem we refer to Alavi, Malde, Schwenk, and Erd\H{o}s~\cite{AlaviEtAl1987}.

The last two years have produced a sharp new picture for tree independence polynomials. Kadrawi and Levit proved that trees of order $26$ already fail log-concavity and organized the first infinite reservoirs $T_{3,m,n}$ and $T_{3,m,n}^{\ast}$~\cite{KadrawiLevit2025}. Galvin constructed a family whose first break can occur arbitrarily close to the top of the coefficient sequence~\cite{Galvin2025}. Bautista-Ramos first showed that arbitrarily many breaks are possible~\cite{BautistaRamos2025}. Very recently, Bautista-Ramos, Guill\'en-Galv\'an, and G\'omez-Salgado reorganized many known tree counterexamples into a common pattern-graph framework, derived linear recurrences for those families, proved infinite families with two and three consecutive breaks together with finite families with four and five consecutive breaks, and identified the circle $|x+1/3|=1/3$ as the non-isolated zero locus for the basic $P_2$-based reservoirs~\cite{BautistaRamosGuillenGalvanGomezSalgado2026}. Ramos and Sun used AI-assisted search to produce many further finite examples~\cite{RamosSun2025}. On the other hand, Li proved in 2026 that the two basic Kadrawi--Levit families are nevertheless always unimodal~\cite{Li2026}. This combination of results naturally suggests a new target: not trees any longer, but unicyclic graphs obtained from those trees by a single edge insertion.

The basic operation of the present paper is therefore the \emph{one-edge closure}
\[
T\longmapsto T+uv,
\]
where $T$ is a tree and $uv$ is a nonedge of $T$. The resulting graph is unicyclic. Our main goal is to use this closure operation to construct serious families of unicyclic graphs whose independence polynomials are still unimodal but no longer log-concave. In particular, we want an approach that handles both symbolic infinite families and computationally certified finite ranges.

\medskip
\begin{targettheorem}[KL-closure theorem]
For $r\in\{0,1,2\}$ and $k\ge 3$ (with $k\ge 4$ when $r=2$), let $U_{k,r}=U_{k,k+r}$ be the unicyclic graph obtained from $T_{3,k,k+r}$ by adding an edge between the exceptional branch vertex $v_1$ and the subdivision vertex $a_{2,1}$ on one fixed $P_2$-arm $v_2a_{2,1}b_{2,1}$ in the $k$-bundle at $v_2$. Then the independence polynomial $I(U_{k,r};x)$ is unimodal but not log-concave. More precisely, the penultimate log-concavity inequality fails.
\end{targettheorem}

The paper is organized around this short statement, and Theorem~\ref{thm:Uk-r-unimodal} proves it in full. The proof separates into three pieces. Theorem~\ref{thm:Uk-r-nlc} gives the symbolic non-log-concavity half. Theorem~\ref{thm:H-mode} proves that the main summand
\[
H_{k,r}(x)=G_k(x)F_{k+r}(x)
\]
is unimodal with a controlled mode. Proposition~\ref{prop:HE-bridge} then uses Darroch localization to pass from that mode statement to unimodality of the full KL-closure polynomial.

Two ideas drive the paper.
\begin{enumerate}[label=(\arabic*),leftmargin=1.7em]
    \item \textbf{Operations first.} We isolate the exact algebra of closure: the standard rooted-product identity in the notation needed here, Darroch mode localization together with an adjacent-mode bridge lemma for the basic real-rooted branch factors, an exact subtraction identity for one-edge insertion, a cycle-transfer recurrence, and a parity analysis of the reduced forest $T-\N[u]-\N[v]$.
    \item \textbf{Families first.} We organize the problem around the four tree reservoirs now known to feed non-log-concavity: the KL families, Galvin's late-break family, the Ramos--Sun finite corpus, and the Bautista--Ramos pattern-graph reservoirs, including the new consecutive-break families.
\end{enumerate}
The payoff is a program in which the main constructions become explicit coefficient calculations rather than an undefined search for examples.

A central structural split is the parity of $\dist_T(u,v)$. Closing an even-length $u$--$v$ path produces an odd cycle, while closing an odd-length path produces an even cycle. Because trees are bipartite, this is equivalent to asking whether the added edge joins the same or opposite parts of the tree bipartition. On the even-cycle side, the closure remains bipartite, so the Levit--Mandrescu upper-third monotonicity theorem gives rigorous tail control~\cite{LevitMandrescu2006}. This is only partial control, since bipartite graphs need not be unimodal in general~\cite{BhattacharyyaKahn2013}. On the odd-cycle side, and more generally on explicit KL closures, the natural tool is a Li-style direct coefficient comparison~\cite{Li2026}.

\medskip
\noindent\textbf{What the paper proves.}
The present manuscript combines the current structural kernel with a stronger family-level package. The proved results are:
\begin{enumerate}[label=(\roman*),leftmargin=1.6em]
    \item exact transport identities for one-edge closure and for weighted cycle closure;
    \item a parity dictionary and a path-with-branches description of the reduced forest controlling the correction term;
    \item an explicit infinite KL-derived unicyclic family $U_{k,r}$, $r\in\{0,1,2\}$, whose independence polynomials are unimodal and non-log-concave for all admissible parameters;
    \item a verified-range sharpening for $U_{k,r}$, proving a unique penultimate break together with exact mode formulas for all admissible $k\le 400$;
    \item a universal exhaustive theorem for substantial Galvin and Bautista--Ramos ranges;
    \item an exhaustive theorem showing that tested KL enlargements always remain unimodal, even though universal preservation of non-log-concavity fails there.
\end{enumerate}
Thus the paper does not merely propose a strategy. It now proves the flagship KL-closure theorem for all admissible parameters, sharpens it on a large exact range, and produces broad finite certified regions of unimodal but non-log-concave unicyclic graphs.

\section{General closure machinery}

\subsection{Rooted attachment products (background)}

Let $(H,r)$ be a rooted graph. Define
\[
A_H(x):=I(H-r;x),\qquad B_H(x):=x\,I(H-\N[r];x).
\]
If $G$ is any graph, let $G\odot(H,r)$ denote the graph obtained by taking one copy of $G$ and, for each vertex $v\in V(G)$, attaching a fresh copy of $H-r$ to $v$ along the root interface.

Before the genuinely new closure identities, we record the standard rooted-product formula in the notation used later. It goes back to Gutman's compound-graph identity and Rosenfeld's rooted-product formula, and also appears as a special case of Zhu's clique-cover product theorem~\cite{Gutman1992,Rosenfeld2010,Zhu2016}.

\begin{proposition}[known rooted-product formula]\label{prop:rooted-attachment}
If $G$ has order $n$, then
\[
I\bigl(G\odot(H,r);x\bigr)=A_H(x)^n\,I\!\left(G;\frac{B_H(x)}{A_H(x)}\right).
\]
\end{proposition}

This identity is used only as bookkeeping for repeated branch attachments. It is the substitution mechanism behind the branch-bundle families studied later.

\subsection{Exact transport for one-edge closures}

Let $T$ be a tree and let $u,v\in V(T)$ be distinct nonadjacent vertices. Set
\[
H=T+uv.
\]
Since $T$ contains a unique $u$--$v$ path, the graph $H$ is unicyclic.

\begin{proposition}[edge-addition identity]\label{prop:edge-addition}
For every tree $T$ and every nonadjacent pair $u,v\in V(T)$,
\[
I(H;x)=I(T;x)-x^2I\bigl(T-\N[u]-\N[v];x\bigr).
\]
\end{proposition}

\begin{proof}
The independent sets of $H=T+uv$ are exactly the independent sets of $T$ except those containing both $u$ and $v$. Any independent set of $T$ containing both $u$ and $v$ must avoid every vertex of $\N(u)\cup\N(v)$, and after deleting $u$ and $v$ it becomes an independent set of $T-\N[u]-\N[v]$. Conversely, every independent set of $T-\N[u]-\N[v]$ extends uniquely to an independent set of $T$ containing both $u$ and $v$. Hence the deleted family contributes exactly the displayed correction term.
\end{proof}

\begin{corollary}[coefficient transport]\label{cor:coeff-transport}
Write
\[
I(T;x)=\sum_{k\ge 0} t_kx^k,\qquad
I\bigl(T-\N[u]-\N[v];x\bigr)=\sum_{k\ge 0} r_kx^k,
\]
where $r_k=0$ for $k<0$. If
\[
I(H;x)=\sum_{k\ge 0} h_kx^k,
\]
then for every $k\ge 0$ one has
\[
h_k=t_k-r_{k-2}.
\]
\end{corollary}

\begin{proof}
This is the coefficientwise form of Proposition~\ref{prop:edge-addition}.
\end{proof}

To track log-concavity locally, define
\[
\Delta_k(G):=s_k(G)^2-s_{k-1}(G)s_{k+1}(G).
\]
Then $I(G;x)$ fails log-concavity at $k$ exactly when $\Delta_k(G)<0$.

\begin{proposition}[exact defect transport]\label{prop:defect-transport}
With the notation above, for every $k\ge 0$ we have
\begin{align*}
\Delta_k(H)
&=\Delta_k(T)-2t_kr_{k-2}+t_{k-1}r_{k-1}+t_{k+1}r_{k-3}
  +r_{k-2}^2-r_{k-3}r_{k-1}.
\end{align*}
\end{proposition}

\begin{proof}
Substituting $h_k=t_k-r_{k-2}$ into the definition of $\Delta_k(H)$ and expanding gives
\begin{align*}
\Delta_k(H)
&=(t_k-r_{k-2})^2-(t_{k-1}-r_{k-3})(t_{k+1}-r_{k-1})\\
&=t_k^2-t_{k-1}t_{k+1}-2t_kr_{k-2}+t_{k-1}r_{k-1}+t_{k+1}r_{k-3}
  +r_{k-2}^2-r_{k-3}r_{k-1},
\end{align*}
and the first two terms equal $\Delta_k(T)$.
\end{proof}

\begin{corollary}[support window]\label{cor:support-window}
For a fixed closure $H=T+uv$, the only indices $k$ at which the local defect can change are those for which at least one of $r_{k-3},r_{k-2},r_{k-1}$ is nonzero. In particular, if
\[
r_{k-3}=r_{k-2}=r_{k-1}=0,
\]
then $\Delta_k(H)=\Delta_k(T)$.
\end{corollary}

\begin{proof}
This is immediate from Proposition~\ref{prop:defect-transport}.
\end{proof}

\begin{remark}[where the new algebra lives]\label{rem:new-algebra}
The reduced forest
\[
F_{u,v}:=T-\N[u]-\N[v]
\]
is the sole source of new algebra after edge insertion. In particular, any family theorem must identify how the coefficient sequence of $I(F_{u,v};x)$ interacts with that of $I(T;x)$. This separates two possible mechanisms: persistence of a pre-existing tree break, and creation of a genuinely new break by closure.
\end{remark}

\subsection{Cycle-transfer recurrence}

The subtraction identity is the cleanest universal formula for one-edge closure, but for exhaustive calculations it is often useful to work directly on the unique cycle together with its attached rooted branches.

\begin{proposition}[cycle-closure transfer formula]\label{prop:cycle-transfer}
Let $U=T+uv$ be a one-edge closure of a tree. Let
\[
C=(c_0,c_1,\dots,c_{\ell-1})
\]
be the unique cycle of $U$, where $c_0=u$ and $c_{\ell-1}=v$, listed cyclically. Deleting the cycle edges leaves rooted trees $B_i$ attached at $c_i$. For each $i$, define
\[
A_i(x)=I(B_i-c_i;x),\qquad B_i(x)=x\,I(B_i-\N[c_i];x).
\]
Set
\[
f_0^{(0)}=A_0,\quad g_0^{(0)}=0,\qquad f_0^{(1)}=0,\quad g_0^{(1)}=B_0,
\]
and for $1\le i\le \ell-1$ define
\[
f_i^{(\varepsilon)}=(f_{i-1}^{(\varepsilon)}+g_{i-1}^{(\varepsilon)})A_i,
\qquad g_i^{(\varepsilon)}=f_{i-1}^{(\varepsilon)}B_i,
\qquad \varepsilon\in\{0,1\}.
\]
Then
\[
I(U;x)=\bigl(f_{\ell-1}^{(0)}+g_{\ell-1}^{(0)}\bigr)+f_{\ell-1}^{(1)}.
\]
If all branch pairs are identical, that is $A_i=A$ and $B_i=B$ for every $i$, then
\[
I(U;x)=A(x)^{\ell} I\!\left(C_{\ell};\frac{B(x)}{A(x)}\right).
\]
\end{proposition}

\begin{proof}
The recurrence counts weighted independent sets along the path $c_0,c_1,\dots,c_{\ell-1}$ under the condition that $c_0$ is excluded or included. If $c_0$ is excluded, then the last cycle vertex may be either excluded or included. If $c_0$ is included, then $c_{\ell-1}$ must be excluded because it is adjacent to $c_0$. Summing these two cases gives the first formula. In the identical-branch case, every independent set $S$ of the cycle contributes $B(x)^{|S|}A(x)^{\ell-|S|}$, and summing over all $S\in\Ind(C_{\ell})$ gives the factorization.
\end{proof}

\subsection{Parity bookkeeping for the unique cycle}

Because every tree is bipartite, the parity split for one-edge closures admits a clean structural reformulation.

\begin{proposition}[parity dictionary]\label{prop:parity-dictionary}
Let $T$ be a tree with bipartition $V(T)=A\cup B$, and let $u,v$ be distinct nonadjacent vertices. Put $H=T+uv$. Then the following are equivalent:
\begin{enumerate}[label=(\roman*),leftmargin=1.6em]
    \item $u$ and $v$ lie in the same part of the bipartition of $T$;
    \item $\dist_T(u,v)$ is even;
    \item the unique cycle of $H$ has odd length;
    \item $H$ is non-bipartite.
\end{enumerate}
Likewise, the following are equivalent:
\begin{enumerate}[label=(\roman*),leftmargin=1.6em]
    \item $u$ and $v$ lie in opposite bipartition classes of $T$;
    \item $\dist_T(u,v)$ is odd;
    \item the unique cycle of $H$ has even length;
    \item $H$ is bipartite.
\end{enumerate}
\end{proposition}

\begin{proof}
In a bipartite graph, two vertices lie in the same part exactly when every path joining them has even length, and in opposite parts exactly when every such path has odd length. Since $T$ is a tree, there is a unique $u$--$v$ path, so the parity of $\dist_T(u,v)$ is determined by the bipartition. Adding the edge $uv$ closes this path into the unique cycle of $H$, whose length is $\dist_T(u,v)+1$. Finally, a graph is bipartite if and only if it contains no odd cycle.
\end{proof}

\begin{table}[t]
\centering
\renewcommand{\arraystretch}{1.18}
\begin{tabular}{>{\raggedright\arraybackslash}p{0.28\textwidth} >{\raggedright\arraybackslash}p{0.22\textwidth} >{\raggedright\arraybackslash}p{0.20\textwidth} >{\raggedright\arraybackslash}p{0.17\textwidth}}
\toprule
added edge joins & parity of $\dist_T(u,v)$ & unique cycle in $T+uv$ & closure type\\
\midrule
same bipartition class & even & odd & non-bipartite\\
opposite bipartition classes & odd & even & bipartite\\
\bottomrule
\end{tabular}
\caption{The parity split can be read directly from the tree bipartition.}
\label{tab:parity-split}
\end{table}

\begin{figure}[t]
\centering
\begin{tikzpicture}[scale=0.95, every node/.style={font=\small}]
    \begin{scope}
        \node at (2.5,2.4) {\textbf{same part}};
        \foreach \x/\shade/\lab in {
            0/blue!18/$u$,
            1/white/{},
            2/blue!18/{},
            3/white/{},
            4/blue!18/$v$}
            \node[circle,draw,fill=\shade,minimum size=9mm,inner sep=0pt] (L\x) at (\x*1.25,0) {\lab};
        \foreach \a/\b in {0/1,1/2,2/3,3/4}
            \draw[line width=0.7pt] (L\a)--(L\b);
        \draw[red!75!black,line width=1pt] (L0) to[out=55,in=125] node[midway,above] {$uv$} (L4);
        \node at (2.5,-1.0) {\textbf{same part $\Rightarrow$ odd cycle}};
    \end{scope}

    \begin{scope}[xshift=8.3cm]
        \node at (1.9,2.4) {\textbf{opposite parts}};
        \foreach \x/\shade/\lab in {
            0/blue!18/$u$,
            1/white/{},
            2/blue!18/{},
            3/white/$v$}
            \node[circle,draw,fill=\shade,minimum size=9mm,inner sep=0pt] (R\x) at (\x*1.25,0) {\lab};
        \foreach \a/\b in {0/1,1/2,2/3}
            \draw[line width=0.7pt] (R\a)--(R\b);
        \draw[red!75!black,line width=1pt] (R0) to[out=55,in=125] node[midway,above] {$uv$} (R3);
        \node at (1.9,-1.0) {\textbf{opposite parts $\Rightarrow$ even cycle}};
    \end{scope}
\end{tikzpicture}
\caption{The parity split for one-edge closures of trees. Vertex shading indicates the bipartition of the underlying tree.}
\label{fig:parity-cartoon}
\end{figure}
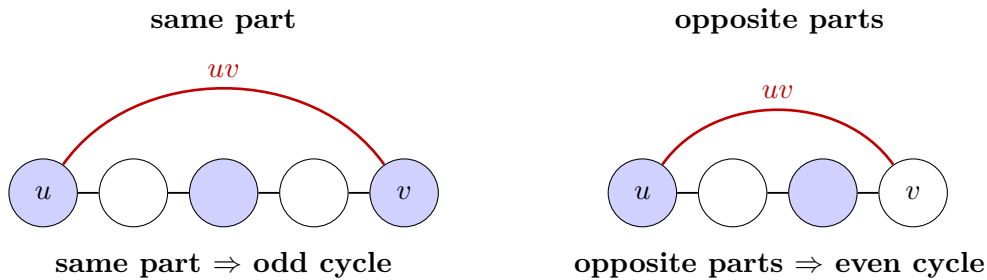

\begin{remark}[why parity matters]\label{rem:parity-matters}
The parity split is simultaneously a split by distance parity in the tree, by parity of the unique cycle after closure, and by whether the closure remains bipartite. Any family calculation should therefore distinguish odd-cycle and even-cycle closures from the outset.
\end{remark}

\subsection{The reduced forest and branch data}

Let
\[
P(u,v)=x_0x_1\cdots x_d
\]
be the unique $u$--$v$ path in $T$, where $x_0=u$ and $x_d=v$. The reduced forest
\[
F_{u,v}:=T-\N[u]-\N[v]
\]
controls both the correction term in Proposition~\ref{prop:edge-addition} and the defect update in Proposition~\ref{prop:defect-transport}.

\begin{lemma}[surviving path core]\label{lem:surviving-path-core}
Assume $d\ge 4$. Then the vertices of $P(u,v)$ that survive in $F_{u,v}$ are exactly
\[
x_2,x_3,\dots,x_{d-2},
\]
and they induce a path on $d-3$ vertices. Consequently,
\begin{itemize}[leftmargin=1.6em]
    \item if $d$ is even, then the surviving path core has odd order and hence a unique central vertex;
    \item if $d$ is odd, then the surviving path core has even order and hence a central edge.
\end{itemize}
\end{lemma}

\begin{proof}
Deleting $\N[u]$ removes $x_0=u$ and $x_1$, while deleting $\N[v]$ removes $x_d=v$ and $x_{d-1}$. No other path vertex belongs to $\N[u]\cup\N[v]$, so the surviving path vertices are precisely $x_2,\dots,x_{d-2}$. The induced subgraph on these vertices is the path inherited from $P(u,v)$.
\end{proof}

\begin{lemma}[branch decomposition of the reduced forest]\label{lem:branch-decomposition}
Every component of
\[
F_{u,v}-\{x_2,\dots,x_{d-2}\}
\]
meets the surviving path core in at most one vertex. Equivalently, $F_{u,v}$ is a path-with-branches forest obtained from the core
\[
x_2x_3\cdots x_{d-2}
\]
by attaching rooted trees to some of its vertices.
\end{lemma}

\begin{proof}
Every vertex of $T$ outside $P(u,v)$ has a unique nearest vertex on that path. After deleting $\N[u]\cup\N[v]$, any surviving component away from the core can still meet the remaining path in at most one vertex; otherwise two distinct contacts would create a cycle in the original tree. Hence each off-core component is attached to a unique surviving path vertex.
\end{proof}

\begin{remark}[where family dependence begins]\label{rem:local-degree-data}
Parity alone determines only whether the surviving core has a central vertex or a central edge. The coefficients of $I(F_{u,v};x)$, and hence the sign and size of the correction in Proposition~\ref{prop:defect-transport}, depend on the rooted branch data attached to that core. This is why the parity split is rigorous while any directional bias must remain family-dependent until branch polynomials are computed.
\end{remark}

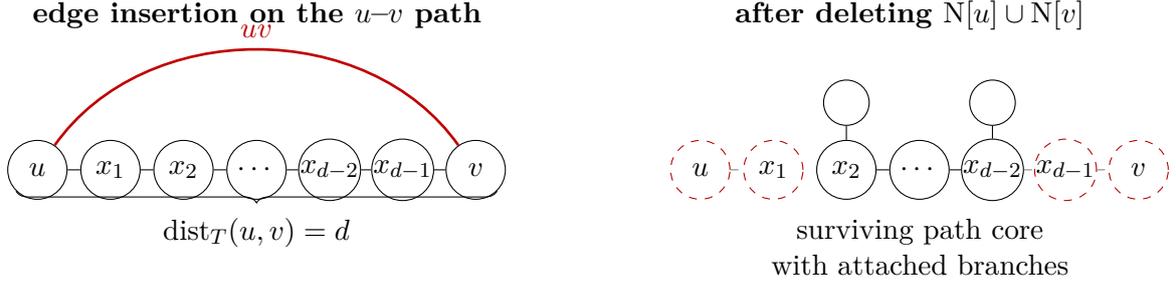
\begin{figure}[t]
\centering
\begin{tikzpicture}[scale=0.89, every node/.style={font=\small}]
    \begin{scope}
        \node at (3.25,2.3) {\textbf{edge insertion on the $u$--$v$ path}};
        \foreach \x/\lab in {0/$u$,1/$x_1$,2/$x_2$,3/$\cdots$,4/$x_{d-2}$,5/$x_{d-1}$,6/$v$}
            \node[circle,draw,minimum size=7.8mm,inner sep=0pt] (A\x) at (\x*1.08,0) {\lab};
        \foreach \a/\b in {0/1,1/2,2/3,3/4,4/5,5/6}
            \draw (A\a)--(A\b);
        \draw[red!75!black,line width=1pt] (A0) to[out=55,in=125] node[midway,above] {$uv$} (A6);
        \draw[decorate,decoration={brace,mirror,amplitude=4.5pt}] (A0.south west) -- (A6.south east)
            node[midway,below=6pt] {$\dist_T(u,v)=d$};
    \end{scope}

    \begin{scope}[xshift=9.8cm]
        \node at (3.1,2.3) {\textbf{after deleting $\N[u]\cup\N[v]$}};
        \node[circle,draw=red!70!black,dashed,minimum size=7.8mm,inner sep=0pt] (B0) at (0,0) {$u$};
        \node[circle,draw=red!70!black,dashed,minimum size=7.8mm,inner sep=0pt] (B1) at (1.08,0) {$x_1$};
        \node[circle,draw,minimum size=7.8mm,inner sep=0pt] (B2) at (2.16,0) {$x_2$};
        \node[circle,draw,minimum size=7.8mm,inner sep=0pt] (B3) at (3.24,0) {$\cdots$};
        \node[circle,draw,minimum size=7.8mm,inner sep=0pt] (B4) at (4.32,0) {$x_{d-2}$};
        \node[circle,draw=red!70!black,dashed,minimum size=7.8mm,inner sep=0pt] (B5) at (5.40,0) {$x_{d-1}$};
        \node[circle,draw=red!70!black,dashed,minimum size=7.8mm,inner sep=0pt] (B6) at (6.48,0) {$v$};
        \draw[gray,dashed] (B0)--(B1);
        \draw (B2)--(B3)--(B4);
        \draw[gray,dashed] (B4)--(B5)--(B6);
        \node[circle,draw,minimum size=6mm,inner sep=0pt] (C1) at (2.16,1.0) {};
        \node[circle,draw,minimum size=6mm,inner sep=0pt] (C2) at (4.32,1.0) {};
        \draw (B2)--(C1);
        \draw (B4)--(C2);
        \node[align=center] at (3.25,-1.15) {surviving path core\\with attached branches};
    \end{scope}
\end{tikzpicture}
\caption{The reduced forest $F_{u,v}$: surviving core $x_2,\dots,x_{d-2}$ with rooted branches.}
\label{fig:reduced-forest}
\end{figure}

\begin{example}[closing a path into a cycle]\label{ex:path-to-cycle}
Take $T=P_n$ and let $u,v$ be its endpoints. Then $H=T+uv=C_n$, while
\[
F_{u,v}=P_{n-4}\qquad (n\ge 4).
\]
Hence Proposition~\ref{prop:edge-addition} becomes
\[
I(C_n;x)=I(P_n;x)-x^2I(P_{n-4};x).
\]
This is the barest model for more elaborate KL-type closures: the entire correction is carried by the surviving path core, with no side branches at all.
\end{example}

\subsection{Coefficient tools, Darroch localization, bridge lemmas, and Li-style reduction}

\begin{lemma}[same-mode sums preserve unimodality]\label{lem:same-mode}
Let $P_1(x),\dots,P_r(x)$ be polynomials with nonnegative coefficients, each unimodal with the same mode index $m$. If $\lambda_1,\dots,\lambda_r\ge 0$, then $\sum_{i=1}^r \lambda_iP_i(x)$ is unimodal with mode $m$.
\end{lemma}

\begin{proof}
Before the common mode $m$, every coefficient sequence is weakly increasing; after $m$, every coefficient sequence is weakly decreasing. Positive linear combinations preserve both monotonicities.
\end{proof}

\begin{proposition}[Ibragimov's strong-unimodality principle]\label{prop:ibragimov}
Let
\[
P(x)=\sum_{i=0}^d p_ix^i,\qquad Q(x)=\sum_{j=0}^e q_jx^j
\]
have nonnegative coefficients. If $(p_0,\dots,p_d)$ is log-concave and has no internal zeros, and if $(q_0,\dots,q_e)$ is unimodal, then the coefficient sequence of $P(x)Q(x)$ is unimodal.
\end{proposition}

\begin{proof}
This is the finite-support form of Ibragimov's strong-unimodality theorem for convolutions of log-concave distributions~\cite{Ibragimov1956}. See also Brenti~\cite[Section~4]{Brenti1989}. We use it below with $P(x)=G_k(x)$ and $Q(x)=F_{k+r}(x)$.
\end{proof}

\begin{lemma}[adjacent-mode bridge]\label{lem:adjacent-mode-bridge}
Let
\[
P(x)=\sum_{j\ge 0} p_jx^j,\qquad Q(x)=\sum_{j\ge 0} q_jx^j
\]
have nonnegative coefficients. Assume that $P$ is unimodal with a mode at $m$, and that $Q$ is unimodal with a mode in $\{m,m+1\}$. Then $P(x)+Q(x)$ is unimodal, with a mode in $\{m,m+1\}$. More precisely, the coefficient sequence of $P+Q$ is weakly increasing through index $m$ and weakly decreasing from index $m+1$ onward, so the only unresolved comparison is between the two bridge coefficients at $m$ and $m+1$.
\end{lemma}

\begin{proof}
For every $j<m$, the unimodality hypotheses give
\[
p_j\le p_{j+1},\qquad q_j\le q_{j+1},
\]
because a mode of $Q$ occurs no earlier than $m$. Hence
\[
p_j+q_j\le p_{j+1}+q_{j+1}\qquad (j<m),
\]
so the coefficient sequence of $P+Q$ is weakly increasing up to index $m$. Likewise, for every $j\ge m+1$ one has
\[
p_j\ge p_{j+1},\qquad q_j\ge q_{j+1},
\]
because $P$ is already nonincreasing after $m$ and $Q$ is nonincreasing after its mode, which lies at $m$ or $m+1$. Therefore
\[
p_j+q_j\ge p_{j+1}+q_{j+1}\qquad (j\ge m+1),
\]
so the coefficient sequence is weakly decreasing from index $m+1$ onward. This proves the claim.
\end{proof}

\begin{lemma}[a tail criterion for multiplication]\label{lem:tail-product}
Let
\[
P(x)=\sum_{i=0}^d p_ix^i,\qquad Q(x)=\sum_{j=0}^e q_jx^j,\qquad R(x)=P(x)Q(x),
\]
with $p_dq_e>0$. Writing $r_k=[x^k]R(x)$, one has
\[
r_{d+e}=p_dq_e,\qquad r_{d+e-1}=p_{d-1}q_e+p_dq_{e-1},
\]
and
\[
r_{d+e-2}=p_{d-2}q_e+p_{d-1}q_{e-1}+p_dq_{e-2}.
\]
In particular, if
\[
\bigl(p_{d-1}q_e+p_dq_{e-1}\bigr)^2< p_dq_e\bigl(p_{d-2}q_e+p_{d-1}q_{e-1}+p_dq_{e-2}\bigr),
\]
then $R$ is non-log-concave.
\end{lemma}

\begin{proof}
The top three coefficients are obtained by collecting the unique degree combinations contributing to them. The displayed inequality is exactly the failure of log-concavity at the index $d+e-1$.
\end{proof}

Darroch's classical theorem says that if
\[
P(x)=\sum_{j=0}^d a_jx^j
\]
has positive coefficients and only real zeros, then every mode $m$ satisfies
\[
\left\lfloor\frac{P'(1)}{P(1)}\right\rfloor\le m\le \left\lceil\frac{P'(1)}{P(1)}\right\rceil.
\]
We use this only as a mode-local guide for the real-rooted auxiliary factors that recur throughout the paper~\cite{Darroch1964}.

\begin{proposition}[Darroch localization for the basic branch factors]\label{prop:darroch-basic}
Let
\[
A_t(x)=(1+2x)^t,\qquad B_t(x)=x(1+x)^t,\qquad E_{k,r}(x)=x(1+2x)^{2k+r+3}.
\]
Then every mode of $A_t$, $B_t$, and $E_{k,r}$ lies respectively in
\[
\left\{\left\lfloor\frac{2t}{3}\right\rfloor,\left\lceil\frac{2t}{3}\right\rceil\right\},\qquad
\left\{\left\lfloor1+\frac{t}{2}\right\rfloor,\left\lceil1+\frac{t}{2}\right\rceil\right\},\qquad
\left\{\left\lfloor\frac{4k+2r+9}{3}\right\rfloor,\left\lceil\frac{4k+2r+9}{3}\right\rceil\right\}.
\]
\end{proposition}

\begin{proof}
The polynomial $A_t(x)$ has only the negative real root $-1/2$, so Darroch's theorem gives the first interval from
\[
\frac{A_t'(1)}{A_t(1)}=\frac{2t}{3}.
\]
For $B_t(x)=x(1+x)^t$, the factor $(1+x)^t$ has only the negative real root $-1$ and
\[
\frac{d}{dx}(1+x)^t\Big|_{x=1}\Big/ (1+1)^t = \frac{t}{2}.
\]
Multiplication by $x$ shifts every mode index up by $1$, giving the second interval. The same shift argument applied to $(1+2x)^{2k+r+3}$ gives the third interval.
\end{proof}

The point is not that Darroch alone proves unimodality here. Rather, Section~4 uses two complementary principles. Ibragimov's theorem handles the main convolution term $H_{k,r}=G_k(x)F_{k+r}(x)$ once a mode is identified, while Darroch, together with Lemma~\ref{lem:adjacent-mode-bridge}, turns the real-rooted correction term into a local bridge device. Once one main summand is known to peak at index $m$ and a Darroch window places the correction term in $\{m,m+1\}$, unimodality of the sum is reduced to the single central comparison between the coefficients at $m$ and $m+1$.

\begin{lemma}[Li-style prefix-plus-tail criterion]\label{lem:prefix-tail}
Let $(a_0,a_1,\dots,a_d)$ be a positive sequence and let $q\in\{1,\dots,d\}$. Assume that
\begin{enumerate}[label=(\arabic*),leftmargin=1.6em]
    \item $a_j^2\ge a_{j-1}a_{j+1}$ for every $1\le j\le q-1$,
    \item $a_q\ge a_{q+1}\ge \cdots\ge a_d$.
\end{enumerate}
Then $(a_0,a_1,\dots,a_d)$ is unimodal.
\end{lemma}

\begin{proof}
Set $\rho_j=a_j/a_{j-1}$ for $1\le j\le q$. Condition~(1) implies
\[
\rho_1\ge\rho_2\ge\cdots\ge\rho_q.
\]
Hence the ratio sequence can cross the level $1$ at most once. Equivalently, there is an index $p\in\{0,1,\dots,q\}$ such that
\[
\rho_1,\dots,\rho_p\ge 1
\qquad\text{and}\qquad
\rho_{p+1},\dots,\rho_q\le 1,
\]
where either block may be empty. Therefore
\[
a_0\le a_1\le\cdots\le a_p
\qquad\text{and}\qquad
a_p\ge a_{p+1}\ge\cdots\ge a_q,
\]
so the prefix $(a_0,\dots,a_q)$ is unimodal. Condition~(2) shows that the tail from $a_q$ onward is weakly decreasing, hence the whole sequence is unimodal.
\end{proof}

\begin{corollary}[bipartite reduction]\label{cor:bipartite-reduction}
Let $G$ be a connected bipartite graph with
\[
I(G;x)=\sum_{j=0}^{\alpha(G)} i_jx^j.
\]
Set
\[
q(G):=\left\lceil\frac{2\alpha(G)-1}{3}\right\rceil.
\]
If
\[
i_j^2\ge i_{j-1}i_{j+1}\qquad (1\le j\le q(G)-1),
\]
then $I(G;x)$ is unimodal.
\end{corollary}

\begin{proof}
By the Levit--Mandrescu last-third theorem for K\"onig-Egerv\'ary graphs~\cite{LevitMandrescu2006}, one has
\[
i_{q(G)}\ge i_{q(G)+1}\ge\cdots\ge i_{\alpha(G)}
\]
for connected bipartite $G$. Hence Lemma~\ref{lem:prefix-tail} applies.
\end{proof}

\section{Tree reservoirs, bundle gadgets, and proof tools}

For $t\ge 0$, let $S_t$ denote the rooted bundle obtained from a root $r$ by attaching $t$ copies of $P_2$ to $r$. The corresponding rooted pair is
\[
A_t(x)=(1+2x)^t,\qquad B_t(x)=x(1+x)^t.
\]
These bundles are the atomic building blocks for all families discussed below.

\begin{definition}\label{def:tree-families}
We use the following standard tree families.
\begin{enumerate}[label=(\arabic*),leftmargin=1.8em]
    \item The Kadrawi--Levit family $T_{3,m,n}$ has a root $v_0$ with children $v_1,v_2,v_3$; the vertex $v_1$ carries three $P_2$-arms $v_1a_{1,j}b_{1,j}$ ($1\le j\le 3$), $v_2$ carries $m$ $P_2$-arms $v_2a_{2,j}b_{2,j}$ ($1\le j\le m$), and $v_3$ carries $n$ $P_2$-arms $v_3a_{3,j}b_{3,j}$ ($1\le j\le n$).
    \item The starred family $T^{\ast}_{3,m,n}$ is obtained from $T_{3,m,n}$ by replacing one of the three $P_2$-arms at $v_1$ by a pendant $P_4$.
    \item Galvin's family $T_{m,t}$ has a root $v$, children $w_1,\dots,w_m$, and each $w_i$ carries $t$ $P_2$-arms.
    \item The Bautista--Ramos family $T_{G,m,t}$ is formed from $m$ disjoint copies of $T_{3,t}$ by attaching their roots to a new root $v_0$ and adding one extra leaf adjacent to $v_0$.
\end{enumerate}
\end{definition}

\begin{remark}[current tree reservoirs]\label{rem:current-reservoirs}
There are four especially natural inputs for the unicyclic program.
\begin{enumerate}[label=(\arabic*),leftmargin=1.8em]
    \item \textbf{The KL reservoir.} The first order-$26$ counterexamples expand to the infinite families $T_{3,m,n}$ and $T_{3,m,n}^{\ast}$~\cite{KadrawiLevit2025,Li2026}. Because Li proved both families unimodal, they are the most promising source of explicit infinite unicyclic examples.
    \item \textbf{The Galvin reservoir.} Galvin's family carries very late breaks of log-concavity~\cite{Galvin2025} and is therefore the natural source for closures that preserve or shift late defects.
    \item \textbf{The Ramos--Sun reservoir.} The AI-generated examples of Ramos and Sun supply a large finite seed bank on $27$--$101$ vertices~\cite{RamosSun2025}. Their role is to feed small sporadic unicyclic examples and low-order conjecture mining.
    \item \textbf{The Bautista--Ramos pattern-graph reservoir.} Bautista-Ramos first produced infinite families with arbitrarily many breaks~\cite{BautistaRamos2025}. The newer paper of Bautista-Ramos, Guill\'en-Galv\'an, and G\'omez-Salgado then identified a common pattern-graph mechanism, derived linear recurrences, proved consecutive-break families, and added a zero-locus viewpoint via Beraha--Kahane--Weiss theory~\cite{BautistaRamosGuillenGalvanGomezSalgado2026}. These are the natural candidates for multi-break and consecutive-break unicyclic families.
\end{enumerate}
\end{remark}

\begin{table}[t]
\centering
\small
\renewcommand{\arraystretch}{1.2}
\begin{tabular}{>{\raggedright\arraybackslash}p{0.17\textwidth} >{\raggedright\arraybackslash}p{0.26\textwidth} >{\raggedright\arraybackslash}p{0.22\textwidth} >{\raggedright\arraybackslash}p{0.24\textwidth}}
\toprule
reservoir & source of tree non-log-concavity & currently available shape information & intended role after one-edge closure\\
\midrule
KL & explicit families $T_{3,m,n}$, $T_{3,m,n}^{\ast}$ built from the first order-$26$ seeds & Li proves unimodality for both families & first explicit infinite unicyclic families; parity frontiers under a canonical closure rule\\
Galvin & infinite family with a break close to $\alpha(T)$ & late-break asymptotics & closures that preserve or shift very late defects\\
Ramos--Sun & many AI-generated finite counterexamples on $27$--$101$ vertices & empirical catalogue rather than symbolic formulas & sporadic finite unicyclic counterexamples; low-order search guidance\\
Bautista--Ramos / pattern-graph & arbitrarily many breaks, then consecutive-break pattern families and linear recurrences & multiple-break and consecutive-break profile; recurrence algebra; zero-locus information for basic $P_2$-based reservoirs & candidate multi-break and consecutive-break unicyclic families\\
\bottomrule
\end{tabular}
\caption{The current tree reservoirs that feed the unicyclic program.}
\label{tab:reservoirs}
\end{table}

\begin{theorem}[upper-third monotonicity on the even-cycle side]\label{thm:upper-third}
Let $H=T+uv$ be a one-edge closure of a tree, and assume $\dist_T(u,v)$ is odd. If
\[
I(H;x)=\sum_{k=0}^{\alpha(H)} h_kx^k,
\]
then
\[
h_{\lceil(2\alpha(H)-1)/3\rceil}\ge h_{\lceil(2\alpha(H)-1)/3\rceil+1}\ge \cdots \ge h_{\alpha(H)}.
\]
\end{theorem}

\begin{proof}
By Proposition~\ref{prop:parity-dictionary}, the graph $H$ is bipartite. Every bipartite graph is K\"onig-Egerv\'ary, so the theorem follows from Levit and Mandrescu's monotonicity theorem for independence polynomials of K\"onig-Egerv\'ary graphs~\cite{LevitMandrescu2006}.
\end{proof}

\begin{corollary}[tail control for even-cycle closures]\label{cor:tail-control}
For an even-cycle closure $H=T+uv$, any failure of unimodality must occur before the final third of the coefficient sequence.
\end{corollary}

\begin{proof}
Theorem~\ref{thm:upper-third} shows that the terminal segment beginning at $\lceil(2\alpha(H)-1)/3\rceil$ is decreasing.
\end{proof}

\begin{remark}[why the tail theorem is only partial]\label{rem:partial-tail}
Theorem~\ref{thm:upper-third} does not prove that every even-cycle closure is unimodal. That distinction is essential: bipartite graphs need not have unimodal independence sequences in general~\cite{BhattacharyyaKahn2013}. The theorem removes the upper tail from the list of possible trouble spots, but it does not settle the full problem.
\end{remark}

\begin{remark}[Li-type calculations on the KL side]\label{rem:li-type}
For KL closures, and especially for odd-cycle closures where no general K\"onig-Egerv\'ary tail theorem applies, the natural unimodality route is a direct coefficient comparison in the spirit of Li's proof for $T_{3,m,n}$ and $T_{3,m,n}^{\ast}$~\cite{Li2026}. The transport formulas above identify exactly which corrected coefficients need to be compared once the polynomial of $F_{u,v}$ is explicit.
\end{remark}

\begin{proposition}[reduction to two symbolic tasks]\label{prop:two-tasks}
Fix a tree family $\mathcal{T}$ together with a canonical rule that selects a nonedge $u_Tv_T$ in each $T\in\mathcal{T}$. Put
\[
H_T=T+u_Tv_T,\qquad F_T=T-\N[u_T]-\N[v_T].
\]
Then the construction of a family of unimodal but non-log-concave unicyclic graphs from $\mathcal{T}$ reduces to the following two tasks:
\begin{enumerate}[label=(\roman*),leftmargin=1.5em]
    \item prove unimodality of $I(H_T;x)$;
    \item use Proposition~\ref{prop:defect-transport} to find an index $k$ with $\Delta_k(H_T)<0$.
\end{enumerate}
Moreover, if $\dist_T(u_T,v_T)$ is odd, then task~\textup{(i)} is already settled on the final third of the coefficient sequence by Theorem~\ref{thm:upper-third}.
\end{proposition}

\begin{proof}
The graph $H_T$ is unicyclic by construction. It is unimodal but not log-concave exactly when conditions \textup{(i)} and \textup{(ii)} both hold. The last sentence is Corollary~\ref{cor:tail-control}.
\end{proof}

\section{The KL-closure family and the flagship theorem}

We now pass from the general closure machinery to the first explicit infinite family. Throughout this section,
\[
F_t(x):=(1+2x)^t+x(1+x)^t,
\]
which should not be confused with the reduced forest notation $F_{u,v}$ from Section~2.

\begin{definition}\label{def:Umn}
For integers $m,n\ge 1$, choose one fixed $P_2$-arm $v_2a_{2,1}b_{2,1}$ in the $m$-bundle at $v_2$ of $T_{3,m,n}$, and let $U_{m,n}$ be the unicyclic graph obtained by adding the edge $v_1a_{2,1}$. For $r\ge 0$ we abbreviate
\[
U_{k,r}:=U_{k,k+r}.
\]
In particular, the symmetric family $U_{k,0}=U_{k,k}$ is the most basic cycle-closure analogue of $T_{3,k,k}$.
\end{definition}

\begin{proposition}[exact subtraction formula]\label{prop:exact-subtraction}
For all integers $m,n\ge 1$,
\[
I(U_{m,n};x)=I(T_{3,m,n};x)-R_{m,n}(x),
\]
where
\[
I(T_{3,m,n};x)=F_3(x)F_m(x)F_n(x)+x(1+2x)^{m+n+3}
\]
and
\[
R_{m,n}(x)=x^2(1+x)^3(1+2x)^{m-1}F_n(x).
\]
\end{proposition}

\begin{proof}
The tree $T_{3,m,n}$ is handled by standard root deletion at $v_0$: if $v_0$ is excluded we obtain the product $F_3F_mF_n$, while if $v_0$ is included each branch contributes a factor $(1+2x)$ for every $P_2$-arm, giving $x(1+2x)^{m+n+3}$. The subtraction term comes from Proposition~\ref{prop:edge-addition} with the chosen nonedge $v_1a_{2,1}$. Deleting the closed neighborhoods of $v_1$ and $a_{2,1}$ leaves three isolated leaves, the $m-1$ remaining $P_2$-arms on the $v_2$-side, and the intact $n$-bundle on the $v_3$-side. This gives the stated formula for $R_{m,n}(x)$.
\end{proof}

\begin{proposition}[reduction to $T_{3,k-1,0}$]\label{prop:reduction-Gk}
For every $k\ge 1$ and $r\ge 0$,
\[
I(U_{k,r};x)=I(T_{3,k-1,0};x)F_{k+r}(x)+x(1+2x)^{2k+r+3}.
\]
Equivalently, if
\[
G_k(x):=I(T_{3,k-1,0};x),\qquad H_{k,r}(x):=G_k(x)F_{k+r}(x),\qquad E_{k,r}(x):=x(1+2x)^{2k+r+3},
\]
then
\[
I(U_{k,r};x)=H_{k,r}(x)+E_{k,r}(x).
\]
Moreover,
\[
G_k(x)=F_4(x)(1+2x)^{k-1}+x(1+x)^kF_3(x).
\]
\end{proposition}

\begin{proof}
Starting from Proposition~\ref{prop:exact-subtraction} with $m=k$ and $n=k+r$, we have
\[
I(U_{k,r};x)=F_3F_kF_{k+r}+x(1+2x)^{2k+r+3}-x^2(1+x)^3(1+2x)^{k-1}F_{k+r}.
\]
Using
\[
F_k=(1+2x)F_{k-1}+x(1+x)^{k-1}(1+x),
\]
we obtain
\[
F_3F_k-x^2(1+x)^3(1+2x)^{k-1}=(1+x)F_3F_{k-1}+x(1+2x)^{k+2}.
\]
The right-hand side is exactly the standard root-deletion formula for $I(T_{3,k-1,0};x)$, because in $T_{3,k-1,0}$ the root $v_0$ has one $3$-arm branch, one $(k-1)$-arm branch, and one leaf branch. Substituting yields the claimed factorization, and expanding once more gives the displayed formula for $G_k(x)$.
\end{proof}

\begin{theorem}[direct log-concavity of $T_{3,k-1,0}$]\label{thm:Gk-logconcave}
For every integer $k\ge 1$, the polynomial
\[
G_k(x)=I(T_{3,k-1,0};x)=F_4(x)(1+2x)^{k-1}+x(1+x)^kF_3(x)
\]
is log-concave.
\end{theorem}

\begin{proof}
Write
\[
G_k(x)=\sum_{j=0}^{k+5} g_{k,j}x^j.
\]
From the explicit expression above one gets
\begin{align*}
g_{k,j}
&=2^j\binom{k-1}{j}+9\,2^{j-1}\binom{k-1}{j-1}+28\,2^{j-2}\binom{k-1}{j-2}
  +38\,2^{j-3}\binom{k-1}{j-3}\\
&\quad +20\,2^{j-4}\binom{k-1}{j-4}+2^{j-5}\binom{k-1}{j-5}
  +\binom{k}{j-1}+7\binom{k}{j-2}+15\binom{k}{j-3}\\
&\quad +11\binom{k}{j-4}+\binom{k}{j-5}.
\end{align*}
For $1\le j\le k+4$, set
\[
u:=j-1,\qquad v:=k+4-j.
\]
Then $u,v\ge 0$, $u+v=k+3$, and a direct symbolic simplification yields
\[
\Delta_{k,j}:=g_{k,j}^2-g_{k,j-1}g_{k,j+1}
=\frac{(u+v-4)!(u+v-3)!}{(u+1)!(u+2)!(v+1)!(v+2)!}\,\Phi(u,v),
\]
where $\Phi(u,v)\in\mathbb{Z}[u,v,2^u]$. Since the factorial prefactor is positive, it is enough to prove $\Phi(u,v)\ge 0$.

For the boundary values $u=0,1,2,3$, the polynomial $\Phi$ factors as follows:
{\small
\begin{align*}
\Phi(0,v)&=512v^2(v-4)(v-3)^2(v-2)^2(v-1)^2(v+1)(2v^2+7v+4),\\
\Phi(1,v)&=1024v^2(v-3)(v-2)^2(v-1)^2(v+1)
\bigl(4v^4+29v^3+98v^2+64v-93\bigr),\\
\Phi(2,v)&=1024v^2(v-2)(v-1)^2(v+1)\bigl(16v^6+204v^5+1111v^4\\
&\qquad\qquad\qquad\qquad\qquad +2784v^3+3661v^2+1332v+5076\bigr),\\
\Phi(3,v)&=1024v^2(v-1)(v+1)\bigl(64v^8+1304v^7+10937v^6+47615v^5\\
&\qquad\qquad\qquad\qquad\qquad +131831v^4+226301v^3+256768v^2+235980v+51120\bigr).
\end{align*}
}
Since $u+v=k+3\ge 4$, the cases $u=0,1,2,3$ force $v\ge 4,3,2,1$ respectively. Hence every displayed factorization is nonnegative.

Assume now that $u\ge 4$ and write $u=s+4$ with $s\ge 0$. Then
\[
\Phi(s+4,v)=\sum_{m=0}^{12} c_m(s)v^m.
\]
A direct coefficient extraction shows that each $c_m(s)$ has the form
\[
c_m(s)=A_m(s)4^s-B_m(s)2^s+C_m(s),
\]
where $A_m,B_m,C_m\in\mathbb{Z}_{\ge 0}[s]$, and moreover the polynomial
\[
\frac{s}{2}A_m(s)-B_m(s)
\]
also has nonnegative coefficients. For example,
\[
c_{12}(s)=1024\cdot 4^s,
\]
and
\[
c_{11}(s)=16\cdot 2^s\bigl(576\cdot 2^s s+2112\cdot 2^s-s^2-3s+10\bigr).
\]
Since $2^s\ge s/2$ for every $s\ge 0$, it follows that
\[
c_m(s)\ge 2^s\left(\frac{s}{2}A_m(s)-B_m(s)\right)+C_m(s)\ge 0.
\]
Hence every coefficient of $\Phi(s+4,v)$ is nonnegative, so $\Phi(s+4,v)\ge 0$ for all $v\ge 0$. Therefore $\Delta_{k,j}\ge 0$ for every $1\le j\le k+4$, and the coefficient sequence of $G_k(x)$ is log-concave.
\end{proof}

\begin{proposition}[the branch factor $F_t$]\label{prop:Ft-logconcave}
For every integer $t\ge 1$, the polynomial
\[
F_t(x)=(1+2x)^t+x(1+x)^t
\]
has a positive log-concave coefficient sequence. Hence it is unimodal. Moreover, for every $t\ge 3$ a mode of $F_t$ is
\[
m_t:=\left\lfloor\frac{2t+1}{3}\right\rfloor.
\]
This mode is unique except when $t=5$, in which case $F_5$ has the two adjacent modes $3$ and $4$. The small cases are
\[
F_1(x)=1+3x+x^2,\qquad F_2(x)=1+5x+6x^2+x^3,
\]
with unique modes $1$ and $2$, respectively.
\end{proposition}

\begin{proof}
Write
\[
F_t(x)=\sum_{j=0}^{t+1} f_{t,j}x^j,\qquad
f_{t,j}=2^j\binom{t}{j}+\binom{t}{j-1},
\]
where the second term is interpreted as $0$ outside its natural range. For $1\le j\le t+1$,
\[
\frac{f_{t,j}}{f_{t,j-1}}
=
\frac{2(t-j+2)\bigl(2^j(t-j+1)+j\bigr)}
{j\bigl(2^j(t-j+2)+2j-2\bigr)}.
\]
Set $s=t-j\ge 0$. A direct subtraction yields
\[
\frac{f_{t,j}}{f_{t,j-1}}-\frac{f_{t,j+1}}{f_{t,j}}
=
\frac{N(j,s)}
{j(j+1)\bigl(2^j s+2^j+j\bigr)\bigl(2^j s+2^{j+1}+2j-2\bigr)},
\]
where
\begin{align*}
N(j,s)
={}&4^j\bigl(2s^3+8s^2+10s+4+6js+4j\bigr)\\
&\quad +2^j j\bigl((2^{j+1}-j+7)s^2+(5j+13)s+6j+6\bigr)
+2j(j+1)(j+s+1).
\end{align*}
Every displayed term is nonnegative for $j\ge 1$ and $s\ge 0$, and in particular $2^{j+1}-j+7>0$. Hence the ratio sequence
\[
\frac{f_{t,1}}{f_{t,0}},\frac{f_{t,2}}{f_{t,1}},\dots,\frac{f_{t,t+1}}{f_{t,t}}
\]
is weakly decreasing. Therefore $(f_{t,0},\dots,f_{t,t+1})$ is log-concave.

It remains to locate a mode for $t\ge 3$. The cases $t=1,2$ were already listed in the statement. Let $m_t=\lfloor(2t+1)/3\rfloor$ and split into residue classes.
If $t=3n$, then $n\ge 1$, $m_t=2n$ and
\begin{align*}
f_{3n,2n}-f_{3n,2n-1}
&=
\frac{\bigl(4^n(n+2)-2n(n-3)\bigr)(3n)!}{(2n)!(n+2)!}>0,\\
f_{3n,2n}-f_{3n,2n+1}
&=
\frac{(n+1)\bigl(4^n+2n-1\bigr)(3n)!}{(2n+1)!(n+1)!}>0.
\end{align*}
If $t=3n+1$, then $n\ge 1$, $m_t=2n+1$ and
\begin{align*}
f_{3n+1,2n+1}-f_{3n+1,2n}
&=
\frac{\bigl(4^n(n+2)-2n^2+3n+2\bigr)(3n+1)!}{(2n+1)!(n+2)!}>0,\\
f_{3n+1,2n+1}-f_{3n+1,2n+2}
&=
\frac{(2\cdot 4^n+n)(3n+1)!}{(2n+1)!(n+1)!}>0.
\end{align*}
If $t=3n+2$, then $n\ge 1$, $m_t=2n+1$ and
\begin{align*}
f_{3n+2,2n+1}-f_{3n+2,2n}
&=
\frac{\bigl(3\cdot 4^n(n+3)-2n^2+5n+3\bigr)(3n+2)!}{(2n+1)!(n+3)!}>0,\\
f_{3n+2,2n+1}-f_{3n+2,2n+2}
&=
\frac{(n-1)(3n+2)!}{(2n+1)!(n+2)!}.
\end{align*}
Thus $m_t$ is always a mode. The only time the second difference vanishes is when $n=1$, i.e. when $t=5$, giving the two adjacent modes $3$ and $4$. In every other case the mode is unique.
\end{proof}

\begin{theorem}[all-parameter non-log-concavity for the KL-closure family]\label{thm:Uk-r-nlc}
Fix $r\in\{0,1,2\}$ and let $U_{k,r}=U_{k,k+r}$ be the unicyclic graph obtained from $T_{3,k,k+r}$ by adding the edge $v_1a_{2,1}$. Write
\[
I(U_{k,r};x)=\sum_{j=0}^{2k+r+6} u_{k,r;j}x^j.
\]
Then:
\begin{enumerate}[label=(\arabic*),leftmargin=1.8em]
    \item for $r=0$ and every $k\ge 3$, the polynomial $I(U_{k,0};x)$ is non-log-concave;
    \item for $r=1$ and every $k\ge 3$, the polynomial $I(U_{k,1};x)$ is non-log-concave;
    \item for $r=2$ and every $k\ge 4$, the polynomial $I(U_{k,2};x)$ is non-log-concave.
\end{enumerate}
In each case, the failure of log-concavity occurs at the penultimate index.
\end{theorem}

\begin{proof}
Set $d:=2k+r+6$. By Proposition~\ref{prop:exact-subtraction}, the top coefficient of $I(U_{k,r};x)$ is
\[
u_{k,r;d}=1,
\]
while the next coefficient is
\[
u_{k,r;d-1}=2^{k-1}+2^{k+r}+2k+r+11.
\]
A straightforward coefficient extraction from Proposition~\ref{prop:exact-subtraction} gives
\begin{align*}
u_{k,r;d-2}
&=2^{k+r}k+11\cdot 2^{k+r}+2^{k-2}k+39\cdot 2^{k-2}+2^{k-1}(k+r)+2^{k+r-1}(k+r)\\
&\quad +17\cdot 2^{2k+r-1}+\frac{k^2}{2}+k(k+r)+\frac{21k}{2}+\frac{(k+r)^2}{2}+\frac{21(k+r)}{2}+15.
\end{align*}
Therefore the penultimate excess is
\[
\Gamma_{k,r}:=u_{k,r;d-2}-u_{k,r;d-1}^2.
\]
For the three values $r=0,1,2$ one obtains the exact formulas
\begin{align*}
4\Gamma_{k,0}&=25\cdot 4^k-(15k+49)2^k-8k^2-92k-424,\\
4\Gamma_{k,1}&=43\cdot 4^k-(25k+107)2^k-8k^2-100k-472,\\
4\Gamma_{k,2}&=55\cdot 4^k-(45k+233)2^k-8k^2-108k-524.
\end{align*}
Now define $\Phi_r(k):=4\Gamma_{k,r}$. A direct calculation yields
\begin{align*}
\Phi_0(k+1)-4\Phi_0(k)&=30k\,2^k+68\cdot 2^k+24k^2+260k+1172>0,\\
\Phi_1(k+1)-4\Phi_1(k)&=50k\,2^k+164\cdot 2^k+24k^2+284k+1308>0,\\
\Phi_2(k+1)-4\Phi_2(k)&=90k\,2^k+376\cdot 2^k+24k^2+308k+1456>0.
\end{align*}
Hence, once $\Phi_r(k)$ is positive, it remains positive for every larger $k$. The initial values are
\[
\Phi_0(3)=76>0,\qquad \Phi_1(3)=452>0,\qquad \Phi_2(4)=6388>0.
\]
Therefore $\Gamma_{k,0}>0$ for all $k\ge 3$, $\Gamma_{k,1}>0$ for all $k\ge 3$, and $\Gamma_{k,2}>0$ for all $k\ge 4$. Since $u_{k,r;d}=1$, this is precisely the failure of log-concavity at the penultimate index.
\end{proof}

\begin{lemma}[central differences of the dominant summand]\label{lem:T1-central}
Let
\[
T_{1,k,r}(x):=F_4(x)(1+2x)^{2k+r-1},\qquad m_{k,r}:=\left\lfloor\frac{4k+2r+8}{3}\right\rfloor,
\]
and define
\[
\begin{aligned}
D^-_{k,r}&:=[x^{m_{k,r}}]T_{1,k,r}(x)-[x^{m_{k,r}-1}]T_{1,k,r}(x),\\
D^+_{k,r}&:=[x^{m_{k,r}}]T_{1,k,r}(x)-[x^{m_{k,r}+1}]T_{1,k,r}(x).
\end{aligned}
\]
Then the following formulas hold for admissible parameters.
\begingroup\small
\begin{enumerate}[label=(\roman*),leftmargin=1.8em]
    \item If $r=0$ and $k=3n$, then $m_{k,r}=4n+2$ and
    \begin{align*}
    D^-_{3n,0}
    &=
    \frac{16^n\bigl(2240n^3+4224n^2+1690n+171\bigr)(6n)!}
    {4(2n+3)!(4n+1)!},\\
    D^+_{3n,0}
    &=
    \frac{16^n\bigl(864n^3+1360n^2+532n+51\bigr)(6n)!}
    {2(4n+3)(2n+2)!(4n+1)!}.
    \end{align*}
    \item If $(r,k)=(0,3n+1)$ or $(r,k)=(2,3n)$, then $m_{k,r}=4n+4$ and
    \begin{align*}
    D^-_{3n+1,0}=D^-_{3n,2}
    &=
    \frac{16^n\bigl(256n^3+448n^2+225n+36\bigr)(6n+1)!}
    {(4n+3)(2n+3)!(4n+1)!},\\
    D^+_{3n+1,0}=D^+_{3n,2}
    &=
    \frac{22\cdot 16^n\bigl(400n^3+744n^2+428n+75\bigr)(6n+1)!}
    {(4n+3)(4n+5)(2n+2)!(4n+1)!}.
    \end{align*}
    \item If $(r,k)=(0,3n+2)$ or $(r,k)=(2,3n+1)$, then $m_{k,r}=4n+5$ and
    \begin{align*}
    D^-_{3n+2,0}=D^-_{3n+1,2}
    &=
    \frac{12\cdot 16^n\bigl(1808n^4+8192n^3+13148n^2+8922n+2175\bigr)(6n+2)!}
    {(4n+3)(4n+5)(2n+4)!(4n+1)!},\\
    D^+_{3n+2,0}=D^+_{3n+1,2}
    &=
    \frac{3\cdot 16^n\bigl(5696n^3+15360n^2+13354n+3765\bigr)(6n+2)!}
    {(4n+3)(4n+5)(2n+3)!(4n+1)!}.
    \end{align*}
    \item If $r=1$ and $k=3n$, then $m_{k,r}=4n+3$ and
    \begin{align*}
    D^-_{3n,1}
    &=
    \frac{4\cdot 16^n\bigl(904n^4+2288n^3+1786n^2+507n+45\bigr)(6n)!}
    {(4n+3)(2n+3)!(4n+1)!},\\
    D^+_{3n,1}
    &=
    \frac{16^n\bigl(2848n^3+3408n^2+1133n+108\bigr)(6n)!}
    {(4n+3)(2n+2)!(4n+1)!}.
    \end{align*}
    \item If $r=1$ and $k=3n+1$, then $m_{k,r}=4n+4$ and
    \begin{align*}
    D^-_{3n+1,1}
    &=
    \frac{3\cdot 16^n\bigl(1120n^3+3792n^2+3797n+1176\bigr)(6n+2)!}
    {(4n+3)(2n+4)!(4n+1)!},\\
    D^+_{3n+1,1}
    &=
    \frac{3\cdot 16^n\bigl(864n^3+2656n^2+2540n+765\bigr)(6n+2)!}
    {(4n+3)(4n+5)(2n+3)!(4n+1)!}.
    \end{align*}
    \item If $r=1$ and $k=3n+2$, then $m_{k,r}=4n+6$ and
    \begin{align*}
    D^-_{3n+2,1}
    &=
    \frac{6\cdot 16^n(3n+2)\bigl(512n^3+1664n^2+1730n+585\bigr)(6n+2)!}
    {(4n+3)(4n+5)(2n+4)!(4n+1)!},\\
    D^+_{3n+2,1}
    &=
    \frac{264\cdot 16^n(3n+2)\bigl(400n^3+1344n^2+1472n+525\bigr)(6n+2)!}
    {(4n+3)(4n+5)(4n+7)(2n+3)!(4n+1)!}.
    \end{align*}
    \item If $r=2$ and $k=3n+2$, then $m_{k,r}=4n+6$ and
    \begin{align*}
    D^-_{3n+2,2}
    &=
    \frac{18\cdot 16^n(3n+2)(6n+5)\bigl(2240n^3+10944n^2+16858n+8325\bigr)(6n+2)!}
    {(4n+3)(4n+5)(2n+5)!(4n+1)!},\\
    D^+_{3n+2,2}
    &=
    \frac{36\cdot 16^n(3n+2)(6n+5)\bigl(864n^3+3952n^2+5844n+2807\bigr)(6n+2)!}
    {(4n+3)(4n+5)(4n+7)(2n+4)!(4n+1)!}.
    \end{align*}
\end{enumerate}
\endgroup
In particular, all these central differences are positive.
\end{lemma}

\begin{proof}
This is a direct coefficient extraction from $T_{1,k,r}(x)=F_4(x)(1+2x)^{2k+r-1}$, followed by routine simplification in the seven residue cases above.
\end{proof}

\begin{theorem}[mode of the main summand]\label{thm:H-mode}
Fix $r\in\{0,1,2\}$ and let $k$ be admissible. Then
\[
H_{k,r}(x)=G_k(x)F_{k+r}(x)
\]
is unimodal and has a mode at
\[
m_{k,r}:=\left\lfloor\frac{4k+2r+8}{3}\right\rfloor.
\]
\end{theorem}

\begin{proof}
By Theorem~\ref{thm:Gk-logconcave}, the coefficient sequence of $G_k(x)$ is positive and log-concave. By Proposition~\ref{prop:Ft-logconcave}, the coefficient sequence of $F_{k+r}(x)$ is unimodal. Hence Proposition~\ref{prop:ibragimov} implies that $H_{k,r}(x)=G_k(x)F_{k+r}(x)$ is unimodal.

To locate a mode, write
\[
H_{k,r}(x)=T_{1,k,r}(x)+J_{k,r}(x),
\]
where
\begin{align*}
T_{1,k,r}(x)&:=F_4(x)(1+2x)^{2k+r-1},\\
J_{k,r}(x)&:=xF_4(x)(1+2x)^{k-1}(1+x)^{k+r}
+xF_3(x)(1+x)^k(1+2x)^{k+r}\\
&\qquad +x^2F_3(x)(1+x)^{2k+r}.
\end{align*}
Since $J_{k,r}(x)$ has nonnegative coefficients, every coefficient of $J_{k,r}$ is at most
\begin{equation}\label{eq:J-at-1}
J_{k,r}(1)=97\,3^{k-1}2^{k+r+1}+35\,2^k3^{k+r}+35\,2^{2k+r}.
\end{equation}
Let $D^-_{k,r}$ and $D^+_{k,r}$ be the two central differences from Lemma~\ref{lem:T1-central}. If
\[
D^-_{k,r}>J_{k,r}(1)\qquad\text{and}\qquad D^+_{k,r}>J_{k,r}(1),
\]
then
\[
[x^{m_{k,r}}]H_{k,r}>[x^{m_{k,r}-1}]H_{k,r},
\qquad
[x^{m_{k,r}}]H_{k,r}>[x^{m_{k,r}+1}]H_{k,r},
\]
so the unimodal coefficient sequence of $H_{k,r}$ has a mode at $m_{k,r}$.

We therefore compare $D^{\pm}_{k,r}$ with $J_{k,r}(1)$. Along a fixed residue class modulo $3$, Lemma~\ref{lem:T1-central} gives explicit formulas for $D^{\pm}_{k,r}$. A direct simplification of those formulas shows that in every residue class
\begin{equation}\label{eq:D-growth}
D^-_{k+3,r}>216\,D^-_{k,r},
\qquad
D^+_{k+3,r}>216\,D^+_{k,r}.
\end{equation}
On the other hand, \eqref{eq:J-at-1} gives
\begin{equation}\label{eq:J-growth}
J_{k+3,r}(1)<216\,J_{k,r}(1)
\end{equation}
for every $k$, because the first two terms in $J_{k,r}(1)$ are multiplied by $216$ when $k$ is replaced by $k+3$, while the last term is multiplied only by $64$.

Now fix a residue class. By \eqref{eq:D-growth} and \eqref{eq:J-growth}, the ratios $D^-_{k,r}/J_{k,r}(1)$ and $D^+_{k,r}/J_{k,r}(1)$ are strictly increasing as $k$ advances by $3$ inside that residue class. Exact evaluation at the following starting values shows where both ratios first exceed $1$:
\[
\renewcommand{\arraystretch}{1.15}
\begin{array}{c|ccc}
 r & k\equiv 0 \pmod 3 & k\equiv 1 \pmod 3 & k\equiv 2 \pmod 3\\
 \hline
 0 & 18 & 22 & 17\\
 1 & 15 & 16 & 23\\
 2 & 21 & 16 & 17
\end{array}
\]
Thus the inequalities $D^-_{k,r}>J_{k,r}(1)$ and $D^+_{k,r}>J_{k,r}(1)$ hold for every admissible $k$ at or beyond the displayed starting value in its residue class. For the finitely many remaining admissible parameters below those starting values, exact coefficient computation from Proposition~\ref{prop:reduction-Gk} shows directly that
\[
[x^{m_{k,r}}]H_{k,r}>[x^{m_{k,r}-1}]H_{k,r},
\qquad
[x^{m_{k,r}}]H_{k,r}>[x^{m_{k,r}+1}]H_{k,r}.
\]
Hence $m_{k,r}$ is a mode of $H_{k,r}(x)$ for every admissible pair $(k,r)$.
\end{proof}

\begin{proposition}[Darroch bridge for the KL closure decomposition]\label{prop:HE-bridge}
Fix $r\in\{0,1,2\}$ and set
\[
m_{k,r}:=\left\lfloor\frac{4k+2r+8}{3}\right\rfloor.
\]
Assume that $H_{k,r}(x)$ is unimodal with a mode at $m_{k,r}$. Then every mode of
\[
E_{k,r}(x)=x(1+2x)^{2k+r+3}
\]
lies in $\{m_{k,r},m_{k,r}+1\}$. Consequently
\[
I(U_{k,r};x)=H_{k,r}(x)+E_{k,r}(x)
\]
is unimodal, with a mode in the same two-point set. In particular, once the mode of $H_{k,r}$ is known, determining the mode of $I(U_{k,r};x)$ reduces to the single bridge comparison between the coefficients at $m_{k,r}$ and $m_{k,r}+1$.
\end{proposition}

\begin{proof}
By Proposition~\ref{prop:darroch-basic}, every mode of $E_{k,r}$ lies in
\[
\left\{\left\lfloor\frac{4k+2r+9}{3}\right\rfloor,\left\lceil\frac{4k+2r+9}{3}\right\rceil\right\}.
\]
Since
\[
m_{k,r}=\left\lfloor\frac{4k+2r+8}{3}\right\rfloor,
\]
a check of the three residue classes of $4k+2r$ modulo $3$ shows that
\[
\left\{\left\lfloor\frac{4k+2r+9}{3}\right\rfloor,\left\lceil\frac{4k+2r+9}{3}\right\rceil\right\}\subseteq \{m_{k,r},m_{k,r}+1\}.
\]
The conclusion therefore follows from Lemma~\ref{lem:adjacent-mode-bridge}.
\end{proof}

\begin{theorem}[KL-closure theorem]\label{thm:Uk-r-unimodal}
Fix $r\in\{0,1,2\}$ and let $U_{k,r}=U_{k,k+r}$ be the unicyclic graph obtained from $T_{3,k,k+r}$ by adding the edge $v_1a_{2,1}$. Then:
\begin{enumerate}[label=(\arabic*),leftmargin=1.8em]
    \item for $r=0$ and every $k\ge 3$, the independence polynomial $I(U_{k,0};x)$ is unimodal but not log-concave;
    \item for $r=1$ and every $k\ge 3$, the independence polynomial $I(U_{k,1};x)$ is unimodal but not log-concave;
    \item for $r=2$ and every $k\ge 4$, the independence polynomial $I(U_{k,2};x)$ is unimodal but not log-concave.
\end{enumerate}
In each case, the penultimate log-concavity inequality fails.
\end{theorem}

\begin{proof}
The penultimate failure is exactly Theorem~\ref{thm:Uk-r-nlc}. By Theorem~\ref{thm:H-mode}, the polynomial $H_{k,r}(x)$ is unimodal with a mode at
\[
m_{k,r}=\left\lfloor\frac{4k+2r+8}{3}\right\rfloor.
\]
Proposition~\ref{prop:HE-bridge} then implies that
\[
I(U_{k,r};x)=H_{k,r}(x)+E_{k,r}(x)
\]
is unimodal. Combining the two statements gives the result.
\end{proof}

\begin{corollary}[verified-range sharpening for the KL-closure family]\label{cor:Uk-r-verified}
Fix $r\in\{0,1,2\}$ and let $\alpha_{k,r}=2k+r+6=\alpha(U_{k,r})$. Then for every admissible $k\le 400$ the following hold.
\begin{enumerate}[label=(\arabic*),leftmargin=1.8em]
    \item Every log-concavity inequality for $I(U_{k,r};x)$ holds except the penultimate one. In particular, the penultimate break is unique in this range.
    \item The polynomial $H_{k,r}(x)=G_k(x)F_{k+r}(x)$ is log-concave.
    \item A mode of $H_{k,r}(x)$ is
    \[
    \mode(H_{k,r})=\left\lfloor\frac{4k+2r+8}{3}\right\rfloor.
    \]
    \item A mode of $E_{k,r}(x)=x(1+2x)^{2k+r+3}$ is
    \[
    \mode(E_{k,r})=\left\lfloor\frac{4k+2r+10}{3}\right\rfloor.
    \]
    \item A mode of $I(U_{k,r};x)$ itself is
    \[
    \mode(I(U_{k,r};x))=\left\lfloor\frac{4k+2r+9}{3}\right\rfloor.
    \]
\end{enumerate}
\end{corollary}

\begin{proof}
For each admissible pair $(k,r)$ with $k\le 400$, the coefficients of $H_{k,r}(x)$, $E_{k,r}(x)$, and $I(U_{k,r};x)$ were computed exactly over the integers from Proposition~\ref{prop:reduction-Gk}. The resulting coefficient sequences were then checked directly. Exact computation shows that every log-concavity inequality for $I(U_{k,r};x)$ holds except the penultimate one, that $H_{k,r}(x)$ is log-concave, and that the displayed mode formulas hold. The theorem already gives unimodality of $I(U_{k,r};x)$ for all admissible parameters.
\end{proof}

\begin{remark}[what remains open]\label{rem:unique-break}
Theorem~\ref{thm:Uk-r-unimodal} settles the flagship existence statement for the KL-closure family. What remains open symbolically is the sharper phenomenon suggested by all exact computations, namely that the penultimate index is the \emph{only} break of log-concavity for every admissible parameter pair $(k,r)$.
\end{remark}

\section{Universal and exhaustive enlargement theorems}

We say that a computational statement is \emph{universal} when it holds for every nonedge of every listed tree. All computations reported here were carried out exactly over the integers using Proposition~\ref{prop:cycle-transfer}; no floating-point arithmetic was used at any stage.

\begin{theorem}[computer-assisted universal theorem for Galvin and Bautista--Ramos families]\label{thm:universal-GBR}
Let
\[
\begin{aligned}
\mathcal{F}={}&\{T_{4,t}:5\le t\le 9\}\cup\{T_{5,t}:5\le t\le 8\}\cup\{T_{6,5},T_{6,6}\}\\
&\cup\{T_{G,2,t}:6\le t\le 8\}\cup\{T_{G,3,t}:6\le t\le 8\}.
\end{aligned}
\]
If $T\in\mathcal{F}$ and $e$ is any nonedge of $T$, then the unicyclic graph $T+e$ has a unimodal and non-log-concave independence polynomial.
\end{theorem}

\begin{proof}
For each tree in the list, every nonedge was enumerated exactly. For each enlargement $T+e$, Proposition~\ref{prop:cycle-transfer} gives the exact independence polynomial, and the resulting coefficient sequence was checked for unimodality and for log-concavity. The results are summarized in Table~\ref{tab:universal-families}. Across all $66{,}303$ enlargements in the table, every coefficient sequence is unimodal and every one fails log-concavity at least once.
\end{proof}

\begin{table}[t]
\centering
\small
\renewcommand{\arraystretch}{1.18}
\begin{tabular}{>{\raggedright\arraybackslash}p{0.41\textwidth} >{\centering\arraybackslash}p{0.20\textwidth} >{\raggedright\arraybackslash}p{0.25\textwidth}}
\toprule
family & enlargements checked & verdict\\
\midrule
$T_{4,t}$, $5\le t\le 9$ & $9{,}170$ & all unimodal, all non-log-concave\\
$T_{5,t}$, $5\le t\le 8$ & $9{,}910$ & all unimodal, all non-log-concave\\
$T_{6,5}$ and $T_{6,6}$ & $5{,}148$ & all unimodal, all non-log-concave\\
$T_{G,2,t}$, $6\le t\le 8$ & $12{,}978$ & all unimodal, all non-log-concave\\
$T_{G,3,t}$, $6\le t\le 8$ & $29{,}097$ & all unimodal, all non-log-concave\\
\midrule
total & $66{,}303$ & exact exhaustive verification\\
\bottomrule
\end{tabular}
\caption{Universal families from Theorem~\ref{thm:universal-GBR}.}
\label{tab:universal-families}
\end{table}

\begin{theorem}[computer-assisted theorem for the tested Kadrawi--Levit families]\label{thm:tested-KL}
Every one-edge enlargement of every tree in the following collections has a unimodal independence polynomial:
\[
T_{3,k,k}\ (4\le k\le 12),\quad T_{3,k,k+1}\ (4\le k\le 9),\quad T_{3,k,k+2}\ (4\le k\le 9),
\]
\[
T^{\ast}_{3,k,k+1}\ (3\le k\le 8),\quad T^{\ast}_{3,k-1,k+1}\ (4\le k\le 9),\quad T^{\ast}_{3,k,k}\ (4\le k\le 9),
\]
\[
T^{\ast}_{3,k,k+3}\ (4\le k\le 9).
\]
However, none of the tested collections enjoys universal preservation of non-log-concavity: in every listed family and every tested range there are one-edge enlargements that are log-concave.
\end{theorem}

\begin{proof}
The proof is again exact and exhaustive. For each listed parameter pair $(m,n)$, all nonedges of $T_{3,m,n}$ or $T^{\ast}_{3,m,n}$ were enumerated, the independence polynomial of every enlargement was computed via Proposition~\ref{prop:cycle-transfer}, and the resulting coefficient sequence was checked directly. The results appear in Table~\ref{tab:tested-KL}. All $34{,}122$ tested enlargements remain unimodal. On the other hand, the last column records $12{,}425$ log-concave enlargements, so universal non-log-concavity fails throughout the tested KL ranges.
\end{proof}

\begin{table}[t]
\centering
\small
\renewcommand{\arraystretch}{1.18}
\begin{tabular}{>{\raggedright\arraybackslash}p{0.48\textwidth} >{\centering\arraybackslash}p{0.19\textwidth} >{\centering\arraybackslash}p{0.20\textwidth}}
\toprule
family and tested range & enlargements checked & log-concave enlargements\\
\midrule
$T_{3,k,k}$, $4\le k\le 12$ & $7{,}860$ & $2{,}296$\\
$T_{3,k,k+1}$, $4\le k\le 9$ & $4{,}136$ & $1{,}306$\\
$T_{3,k,k+2}$, $4\le k\le 9$ & $4{,}586$ & $1{,}822$\\
$T^{\ast}_{3,k,k+1}$, $3\le k\le 8$ & $3{,}710$ & $1{,}216$\\
$T^{\ast}_{3,k-1,k+1}$, $4\le k\le 9$ & $4{,}136$ & $1{,}394$\\
$T^{\ast}_{3,k,k}$, $4\le k\le 9$ & $4{,}136$ & $1{,}211$\\
$T^{\ast}_{3,k,k+3}$, $4\le k\le 9$ & $5{,}558$ & $3{,}180$\\
\midrule
total & $34{,}122$ & $12{,}425$\\
\bottomrule
\end{tabular}
\caption{Tested Kadrawi--Levit families: every enlargement remains unimodal, but many are log-concave.}
\label{tab:tested-KL}
\end{table}

\begin{example}[a representative enlargement]\label{ex:representative}
A representative example is obtained from $T_{3,4,4}$ by adding one suitable edge. Its independence polynomial is
\begin{align*}
I(U;x)&=1+26x+299x^2+2022x^3+9000x^4+27901x^5+62014x^6+99907x^7\\
&\quad +116133x^8+95263x^9+52524x^{10}+17593x^{11}+2783x^{12}+43x^{13}+x^{14}.
\end{align*}
The sequence is unimodal, but log-concavity fails at the penultimate index because
\[
43^2<2783\cdot 1.
\]
This example is typical of the symbolic family $U_{k,r}$: the break occurs very near the top while the global shape remains strongly unimodal.
\end{example}

\section{Parity-frontier program and outlook}

The previous sections already deliver explicit infinite and finite sources of unimodal but non-log-concave unicyclic graphs. The current structural kernel also suggests a sharper program for pushing further.

\begin{conjecture}[parity-separated frontiers inside a fixed reservoir]\label{conj:parity-frontier}
Within any fixed reservoir and any fixed closure rule, the minimal parameter values at which unimodal-but-not-log-concave unicyclic graphs first appear split into two distinct frontier pieces according to the parity of $\dist_T(u,v)$. Equivalently, odd-cycle and even-cycle closures form separate symbolic regimes.
\end{conjecture}

\begin{remark}[the intended role of each reservoir]\label{rem:roles}
The family program naturally divides into four channels.
\begin{enumerate}[label=\textbf{Channel \arabic*.},leftmargin=2.4em]
    \item \textbf{KL closures.} These furnish the first explicit infinite unicyclic families, because the base trees are already symbolic and their unimodality is known.
    \item \textbf{Galvin closures.} These test how late tree defects behave under one-edge closure. They are the natural source for unicyclic families whose breaks remain close to the top of the sequence.
    \item \textbf{Ramos--Sun closures.} These form the finite, computationally guided side of the project. Their role is to produce small sporadic unicyclic examples and to suggest low-order patterns not yet visible in the symbolic families.
    \item \textbf{Bautista--Ramos closures.} These are designed for the multiple-break problem. They should indicate whether one-edge closure can preserve or amplify several distinct failures of log-concavity.
\end{enumerate}
\end{remark}

\begin{problem}[the KL closure theorem]\label{prob:kl-closure}
Fix a canonical closure rule on $T_{3,m,n}$ and on $T^{\ast}_{3,m,n}$. Determine exactly for which parameters the resulting unicyclic graphs are unimodal but not log-concave.
\end{problem}

\begin{problem}[Galvin-to-unicyclic transfer]\label{prob:galvin-transfer}
Construct a closure rule on Galvin's family for which a late break of log-concavity survives in the unicyclic graph. Determine how the break index moves under the transport formula of Proposition~\ref{prop:defect-transport}.
\end{problem}

\begin{problem}[finite sporadic closures from the Ramos--Sun reservoir]\label{prob:ramos-sun}
Use the finite catalogue of Ramos and Sun as a search space for the smallest unicyclic graphs whose independence polynomials are unimodal but not log-concave. Organize the resulting examples by odd-cycle versus even-cycle closure type.
\end{problem}

\begin{problem}[multi-break unicyclic families]\label{prob:multi-break}
Starting from Bautista--Ramos's multi-break families~\cite{BautistaRamos2025} and from the newer pattern-graph families with two and three consecutive breaks~\cite{BautistaRamosGuillenGalvanGomezSalgado2026}, decide whether one-edge closure can produce unicyclic graphs whose independence polynomials fail log-concavity at more than one index, or even at several consecutive indices, while remaining unimodal.
\end{problem}

\begin{problem}[local branch classification]\label{prob:local-branch}
Describe the first extremal closures in terms of the rooted branch data attached to the surviving path core in Lemma~\ref{lem:surviving-path-core}. Equivalently, classify the local degree configurations around the cycle that force the correction term in Proposition~\ref{prop:defect-transport} to become negative.
\end{problem}

\begin{problem}[zero loci under closure]\label{prob:zero-locus}
For the basic $P_2$-based tree reservoirs, Bautista-Ramos, Guill\'en-Galv\'an, and G\'omez-Salgado proved that the non-isolated limit points of the zeros lie on the circle $|x+1/3|=1/3$~\cite{BautistaRamosGuillenGalvanGomezSalgado2026}. Determine whether canonical one-edge closures of those reservoirs admit an analogous Beraha--Kahane--Weiss description, and whether odd-cycle and even-cycle closures preserve or perturb that circle.
\end{problem}

\begin{theorem}[structural kernel for the family program]\label{thm:rigorous-kernel}

The construction of unimodal but non-log-concave one-edge closures of trees rests on the following rigorous inputs:
\begin{enumerate}[label=(\roman*),leftmargin=1.6em]
    \item the edge-addition identity in Proposition~\ref{prop:edge-addition};
    \item the coefficient and defect transport formulas in Corollary~\ref{cor:coeff-transport} and Proposition~\ref{prop:defect-transport};
    \item the cycle-transfer recurrence in Proposition~\ref{prop:cycle-transfer};
    \item the parity dictionary in Proposition~\ref{prop:parity-dictionary};
    \item the surviving-path and branch-decomposition statements in Lemmas~\ref{lem:surviving-path-core} and~\ref{lem:branch-decomposition};
    \item the upper-third monotonicity theorem for even-cycle closures in Theorem~\ref{thm:upper-third}.
\end{enumerate}
Consequently, once a reservoir and a closure rule are fixed, the remaining family-specific work is concentrated in two calculations: the polynomial $I(F_{u,v};x)$ and a unimodality proof for the corrected coefficient sequence on the part not already controlled by \textup{(vi)}.
\end{theorem}

\begin{proof}
Items \textup{(i)}--\textup{(v)} were proved in the preceding sections, and item \textup{(vi)} is Theorem~\ref{thm:upper-third}. The concluding sentence is precisely Proposition~\ref{prop:two-tasks} read together with the description of $F_{u,v}$ supplied by Lemmas~\ref{lem:surviving-path-core} and~\ref{lem:branch-decomposition}.
\end{proof}

Taken together, the paper now has a clear slogan and a clear theorem-led spine. The slogan is that one can close trees into unicyclic counterexamples. The flagship statement is the KL-closure theorem displayed in the introduction, and Theorem~\ref{thm:Uk-r-unimodal} proves it for every admissible parameter. Theorem~\ref{thm:Uk-r-nlc} isolates the symbolic source of non-log-concavity, Theorem~\ref{thm:H-mode} supplies the symbolic mode calculation for the main summand, and Proposition~\ref{prop:HE-bridge} converts that mode statement into unimodality of the full closure polynomial. Corollary~\ref{cor:Uk-r-verified} then sharpens the picture on a large exact range by showing a unique penultimate break and exact mode formulas. Theorems~\ref{thm:universal-GBR} and~\ref{thm:tested-KL} show that the broader reservoir program is already mathematically productive. The remaining frontier is no longer the existence theorem for the KL family, but the sharper all-parameter uniqueness-of-break problem, the transfer of the same method to the wider Galvin, Ramos--Sun, and Bautista--Ramos pattern-graph reservoirs, and the new zero-locus and consecutive-break questions suggested by recent tree-side recurrence theory.

\end{document}